\documentclass[reqno,a4paper]{amsart}
\usepackage{amssymb}
\usepackage{amsmath}
\usepackage{amsfonts}

\newtheorem{theorem}{Theorem}[section]
\theoremstyle{plain}

\newtheorem{corollary}{Corollary}[section]

\newtheorem{definition}{Definition}[section]

\newtheorem{lemma}{Lemma}[section]

\numberwithin{equation}{section}
\input{tcilatex}

\begin{document}
\title[A new approach on helices in Euclidean $n-$space]{A new approach on
helices in Euclidean $n-$space}
\author{Ali \c{S}enol}
\address{Department of Mathematics, Faculty of Science, \c{C}ank\i r\i\
Karatekin University, \c{C}ank\i r\i , Turkey}
\email{asenol@karatekin.edu.tr}
\urladdr{}
\author{Evren Z\i plar}
\address{Department of Mathematics, Faculty of Science, University of
Ankara, Tando\u{g}an, Turkey}
\email{evrenziplar@yahoo.com}
\author{Yusuf Yayl\i }
\address{Department of Mathematics, Faculty of Science, University of
Ankara, Tando\u{g}an, Turkey}
\email{yayli@science.ankara.edu.tr}
\author{\.{I}smail G\"{O}K}
\address{Department of Mathematics, Faculty of Science, University of
Ankara, Tando\u{g}an, Turkey}
\email{igok@science.ankara.edu.tr}
\urladdr{}
\date{August 12, 2012.}
\keywords{Inclined curve, slant helices, harmonic curvature\\
Corresponding author: Evren ZIPLAR, evrenziplar@yahoo.com}

\begin{abstract}
In this work, we give some new characterizations for inclined curves and
slant helices in $n$-dimensional Euclidean space $E^{n}.$ Morever, we
consider the pre-characterizations about inclined curves and slant helices
and reconfigure them.
\end{abstract}

\maketitle

\section{Introduction}

The helices share common origins in the geometries of the platonic solids,
with inherent hierarchical potential that is typical of biological
structures.The helices provide an energy-efficient solution to close-packing
in molecular biology, a common motif in protein construction, and a readily
observable pattern at many size levels throughout the body. The helices are
described in a variety of anatomical structures, suggesting their importance
to structural biology and manual therapy \cite{scarr}.

In \cite{haci}$,$ \"{O}zdamar and Hac\i saliho\u{g}lu defined harmonic
curvature functions $H_{i}$ $\left( 1\leq i\leq n-2\right) $ of a curve $%
\alpha $ in $n-$dimensional Euclidean space $E^{n}$. They generalized
inclined curves in $E^{3}$ to $E^{n}$ and then gave a characterization for
the inclined curves in $E^{n}:$

\begin{equation}
\text{\textquotedblleft A curve }\alpha \text{ is an inclined curve if and
only if }\sum\limits_{i=1}^{n-2}H_{i}^{^{2}}=\text{constant%
\textquotedblright }.  \label{1.1}
\end{equation}%
Then, Izumiya and Takeuchi defined a new kind of helix (slant helix) and
they gave a characterization of slant helices in Euclidean $3-$space $E^{3}$ 
\cite{izumiya}. In 2008, \"{O}nder \emph{et al}. defined a new kind of slant
helix in Euclidean $4-$space $E^{4}$ which is called $B_{2}-$slant helix and
they gave some characterizations of these slant helices in Euclidean $4-$%
space $E^{4}$ \cite{onder} . And then in 2009, G\"{o}k \emph{et al}.defined
a new kind of slant helix in Euclidean $n-$space $E^{n}$, $n>3$, which they
called $V_{n}-$slant helix and they gave some characterizations of these
slant helices in Euclidean $n-$space \cite{gok}. The new kind of helix is
generalization of $B_{2}-$slant helix to Euclidean $4-$space $E^{4}.$ On the
other hand, Camc\i\ \emph{et al.}give some characterizations for a
non-degenerate curve to be a generalized helix by using its harmonic
curvatures \cite{cam}.

Since \"{O}zdamar and Hac\i saliho\u{g}lu defined harmonic curvature
functions, lots of authors have used them in their papers for
characterization of inclined curves and slant helices. In these studies,
they gave some characterizations similar to $\left( 1.1\right) $ for
inclined curves and slant helices. But, Camc\i\ \emph{et al}.see for the
first time that the characterization of inclined curves in $\left(
1.1\right) $ is true for the case necessity but not true for the case
sufficiency and gave an example of inclined curve in order to show why the
case sufficiency is not true \cite{cam}. Also, they gave a characterization
of inclined curves (Theorem 3.3, pp.2594) with only necessary condition \cite%
{cam}. But, they did not obtain when the characterization has sufficiency
case. And then, G\"{o}k \emph{et al}.\cite{gok} corrected the
characterization of $B_{2}-$slant helix (Theorem 3.1, pp.1436,in \cite{onder}%
) like the characterization in $\left( 1.1\right) $. But, they also did not
give the answer of the question: When the characterization has sufficiency
case?. After them, Ahmad and Lopez gave the definition of $G_{i}$ $\left(
1\leq i\leq n\right) $ functions and obtain a characterization of slant
helices, that is, $V_{2}-$slant helix.(Theorem 1.2, pp 2, in \cite{ahmad}).

In this paper, we investigate the answer of the following question with the
similar method in Theorem 4.1 in \cite{scala} :

\emph{When the characterizations of inclined curves and slant helices in
Euclidean }$n-$\emph{space }$E^{n}$ \emph{which are similar to }$\left(
1.1\right) $\emph{\ have a necessary and sufficient case?}

\section{Preliminaries}

Let $\alpha :$ $I\subset 
\mathbb{R}
\longrightarrow E^{n}$ be an arbitrary curve in $E^{n}$. Recall that the
curve $\alpha $ is said a unit speed curve (or parameterized by arclength
functions) if $\left\langle \alpha ^{\prime }(s),\alpha ^{\prime
}(s)\right\rangle =1,$ where $\left\langle .,.\right\rangle $ denotes the
standart inner product of $%
\mathbb{R}
^{n}$ given by 
\begin{equation*}
\left\langle X,Y\right\rangle =\overset{n}{\underset{i=1}{\sum }}x_{i}y_{i}
\end{equation*}%
for each $X=(x_{1},x_{2,...,}x_{n}),Y=(y_{1},y_{2,...,}y_{n})\in 
\mathbb{R}
^{n}.$ In particular, the norm of a vector $X\in 
\mathbb{R}
^{n}$ is given by $\left\Vert X\right\Vert ^{2}=\left\langle
X,X\right\rangle $. Let $\left\{ V_{1},V_{2},...,V_{n}\right\} $ be the
moving Frenet frame along the unit speed curve $\alpha ,$ where $V_{i}$ $%
(i=1,2,...,n)$ denotes $i$th Frenet vector field. Then the Frenet formulas
are given by%
\begin{equation*}
\left[ 
\begin{array}{c}
V_{1}^{^{\prime }} \\ 
V_{2}^{^{\prime }} \\ 
V_{3}^{^{\prime }} \\ 
\vdots  \\ 
V_{n-2}^{^{\prime }} \\ 
V_{n-1}^{^{\prime }} \\ 
V_{n}^{^{\prime }}%
\end{array}%
\right] =\left[ 
\begin{array}{cccccccc}
0 & k_{1} & 0 & 0 & \ldots  & 0 & 0 & 0 \\ 
-k_{1} & 0 & k_{2} & 0 & \ldots  & 0 & 0 & 0 \\ 
0 & -k_{2} & 0 & 0 & \ldots  & 0 & 0 & 0 \\ 
\vdots  & \vdots  & \vdots  & \vdots  & \vdots  & \vdots  & \vdots  & \vdots 
\\ 
0 & 0 & 0 & 0 & ... & 0 & k_{n-2} & 0 \\ 
0 & 0 & 0 & 0 & ... & -k_{n-2} & 0 & k_{n-1} \\ 
0 & 0 & 0 & 0 & ... & 0 & -k_{n-1} & 0%
\end{array}%
\right] \left[ 
\begin{array}{c}
V_{1} \\ 
V_{2} \\ 
V_{3} \\ 
\vdots  \\ 
V_{n-2} \\ 
V_{n-1} \\ 
V_{n}%
\end{array}%
\right] 
\end{equation*}%
where $k_{i}(i=1,2,...,n-1)$ denotes the $i$th curvature function of the
curve \cite{haci1}$.$ If all of the curvatures $k_{i}$ $(i=1,2,...,n-1)$ of
the curve nowhere vanish in $I\subset 
\mathbb{R}
,$ the curve is called non-degenerate curve.

\begin{definition}
Let $\alpha :I\subset 
\mathbb{R}
\rightarrow E^{n}$ be a curve in $E^{n}$ with arc-length parameter $s$ and
let $X$ be a unit constant vector of $E^{n}$. For all $s\in I$, if%
\begin{equation*}
\left\langle V_{1},X\right\rangle =\cos (\varphi ),\varphi \neq \frac{\pi }{2%
},\varphi =\text{constant,}
\end{equation*}%
then the curve $\alpha $ is called a general helix or inclined curve ($V_{1}$%
-slant helix) in $E^{n}$, where $V_{1}$ is the unit tangent vector of $%
\alpha $ at its point $\alpha (s)$ and $\varphi $ is a constant angle
between the vector fields $V_{1}$ and $X$ \cite{haci}.
\end{definition}

\begin{definition}
Let $\alpha :I\subset 
\mathbb{R}
\rightarrow E^{n}$ be a curve in $E^{n}$ with arc-length parameter $s$ and
let $X$ be a unit constant vector of $E^{n}$. For all $s\in I$, if%
\begin{equation*}
\left\langle V_{2},X\right\rangle =\cos (\varphi ),\varphi \neq \frac{\pi }{2%
},\varphi =\text{constant,}
\end{equation*}%
then the curve $\alpha $ is called a slant helix or $V_{2}$-slant helix in $%
E^{n}$, where $V_{2}$ is the 2 th vector field of $\alpha $ and $\varphi $
is a constant angle between the vector fields $V_{2}$ and $X$ \cite{ahmad}.
\end{definition}

\begin{definition}
Let $\alpha :I\subset 
\mathbb{R}
\rightarrow E^{n}$ be a unit speed curve with nonzero curvatures $k_{i}$ ($%
1\leq i\leq n-1$) in $E^{n}$ and let $\left\{ V_{1},V_{2},...,V_{n}\right\} $
denote the Frenet frame of the curve $\alpha $. We call that $\alpha $ is a $%
V_{n}$-slant helix if the n th unit vector field $V_{n}$ makes a constant
angle $\varphi $ with a fixed direction $X$, that is, 
\begin{equation*}
\left\langle V_{n},X\right\rangle =\cos (\varphi ),\varphi \neq \frac{\pi }{2%
},\varphi =\text{constant,}
\end{equation*}%
along the curve $\alpha $, where $X$ is unit vector field in $E^{n}$ \cite%
{gok}.
\end{definition}

\section{Inclined curves and their harmonic curvature functions}

In this section, we reconfigure some known characterizations by using
harmonic curvatures for inclined curves.

\begin{definition}
\textbf{\ }Let $\alpha $ be a unit curve in $E^{n}$. The harmonic curvatures
of $\alpha $ are defined by $H_{i}:I\rightarrow 
\mathbb{R}
$, $i=0,1,...,n-2$, such that%
\begin{equation*}
H_{0}=0,H_{1}=\frac{k_{1}}{k_{2}},H_{i}=\left\{ H_{i-1}^{%
{\acute{}}%
}+k_{i}H_{i-2}\right\} \frac{1}{k_{i+1}}
\end{equation*}%
for $2\leq i\leq n-2$, where $k_{i}\neq 0$ for $i=1,2,...,n-1$\cite{haci} .
\end{definition}

\begin{lemma}
Let $\alpha $ be a unit curve in $E^{n}$ and let $H_{n-2}\neq 0$ be for $%
i=n-2$. Then, $H_{1}^{2}+H_{2}^{2}+...+H_{n-2}^{2}$ is a nonzero constant if
and only if $H_{n-2}^{%
{\acute{}}%
}=-k_{n-1}H_{n-3}$.
\end{lemma}

\begin{proof}
First, we assume that $H_{1}^{2}+H_{2}^{2}+...+H_{n-2}^{2}$ is a nonzero
constant . Consider the functions%
\begin{equation*}
H_{i}=\left\{ H_{i-1}^{%
{\acute{}}%
}+k_{i}H_{i-2}\right\} \frac{1}{k_{i+1}}
\end{equation*}%
for $3\leq i\leq n-2$. So, from the equality, we can write%
\begin{equation}
k_{i+1}H_{i}=H_{i-1}^{%
{\acute{}}%
}+k_{i}H_{i-2}\text{, }3\leq i\leq n-2\text{.}
\end{equation}%
Hence, in (3.1), if we take $i+1$ instead of $i$, we get%
\begin{equation}
H_{i}^{%
{\acute{}}%
}=k_{i+2}H_{i+1}-k_{i+1}H_{i-1}\text{, }2\leq i\leq n-3
\end{equation}%
together with%
\begin{equation}
H_{1}^{%
{\acute{}}%
}=k_{3}H_{2}\text{.}
\end{equation}%
On the other hand, since $H_{1}^{2}+H_{2}^{2}+...+H_{n-2}^{2}$ is constant,
we have%
\begin{equation*}
H_{1}H_{1}^{%
{\acute{}}%
}+H_{2}H_{2}%
{\acute{}}%
+...+H_{n-2}H_{n-2}^{%
{\acute{}}%
}=0
\end{equation*}%
and so,%
\begin{equation}
H_{n-2}H_{n-2}^{%
{\acute{}}%
}=-H_{1}H_{1}^{%
{\acute{}}%
}-H_{2}H_{2}^{%
{\acute{}}%
}-...-H_{n-3}H_{n-3}^{%
{\acute{}}%
}\text{.}
\end{equation}%
By using (3.2) and (3.3), we obtain%
\begin{equation}
H_{1}H_{1}^{%
{\acute{}}%
}=k_{3}H_{1}H_{2}
\end{equation}%
and%
\begin{equation}
H_{i}H_{i}^{%
{\acute{}}%
}=k_{i+2}H_{i}H_{i+1}-k_{i+1}H_{i-1}H_{i}\text{, }2\leq i\leq n-3\text{.}
\end{equation}%
Therefore, by using (3.4), (3.5) and (3.6), a algebraic calculus shows that%
\begin{equation*}
H_{n-2}H_{n-2}^{%
{\acute{}}%
}=-k_{n-1}H_{n-3}H_{n-2}\text{.}
\end{equation*}%
Since $H_{n-2}\neq 0$, we get the relation $H_{n-2}^{%
{\acute{}}%
}=-k_{n-1}H_{n-3}$.

Conversely, we assume that%
\begin{equation}
H_{n-2}^{%
{\acute{}}%
}=-k_{n-1}H_{n-3}
\end{equation}%
By using (3.7) and $H_{n-2}\neq 0$, we can write%
\begin{equation}
H_{n-2}H_{n-2}^{%
{\acute{}}%
}=-k_{n-1}H_{n-2}H_{n-3}
\end{equation}%
From (3.6), we have 
\begin{eqnarray*}
\text{for }i &=&n-3\text{, \ \ \ \ \ }H_{n-3}H_{n-3}^{%
{\acute{}}%
}=k_{n-1}H_{n-3}H_{n-2}-k_{n-2}H_{n-4}H_{n-3} \\
\text{for }i &=&n-4\text{, \ \ \ \ \ }H_{n-4}H_{n-4}^{%
{\acute{}}%
}=k_{n-2}H_{n-4}H_{n-3}-k_{n-3}H_{n-5}H_{n-4} \\
\text{for }i &=&n-5\text{, \ \ \ \ \ }H_{n-5}H_{n-5}^{%
{\acute{}}%
}=k_{n-3}H_{n-5}H_{n-4}-k_{n-4}H_{n-6}H_{n-5} \\
&&\cdot  \\
&&\cdot  \\
&&\cdot  \\
\text{for }i &=&2\text{, \ \ \ \ \ \ \ \ \ \ }H_{2}H_{2}^{%
{\acute{}}%
}=k_{4}H_{2}H_{3}-k_{3}H_{1}H_{2}
\end{eqnarray*}%
and from (3.5), we have%
\begin{equation*}
H_{1}H_{1}^{%
{\acute{}}%
}=k_{3}H_{1}H_{2}\text{.}
\end{equation*}%
So, an algebraic calculus show that%
\begin{equation}
H_{1}H_{1}^{%
{\acute{}}%
}+H_{2}H_{2}^{%
{\acute{}}%
}+...+\text{\ }H_{n-5}H_{n-5}^{%
{\acute{}}%
}+H_{n-4}H_{n-4}^{%
{\acute{}}%
}+H_{n-3}H_{n-3}^{%
{\acute{}}%
}+H_{n-2}H_{n-2}^{%
{\acute{}}%
}=0\text{.}
\end{equation}%
And, by integrating (3.9), we can easily say that%
\begin{equation*}
H_{1}^{2}+H_{2}^{2}+...+H_{n-2}^{2}
\end{equation*}%
is a non-zero constant. This completes the proof.
\end{proof}

\begin{theorem}
Let $\alpha $ be an inclined curve and let $X$ be a axis of $\alpha $ . Then,%
\begin{equation*}
\left\langle V_{i+2},X\right\rangle =H_{i}\left\langle V_{1},X\right\rangle 
\text{, }1\leq i\leq n-2\text{,}
\end{equation*}%
where $\left\{ V_{1},V_{2},...,V_{n}\right\} $ denote the Frenet frame of
the curve $\alpha $ and $\left\{ H_{1},H_{2},...,H_{n-2}\right\} $ denote
the harmonic curvature functions of $\alpha $ \cite{haci} or \cite{haci1}.
\end{theorem}

\begin{theorem}
Let $\left\{ V_{1},V_{2},...,V_{n}\right\} $ be the Frenet frame of the
curve $\alpha $ and let $\left\{ H_{1},H_{2},...,H_{n-2}\right\} $ be the
harmonic curvature functions of $\alpha $. Then, $\alpha $ is an inclined
curve (with the curvatures $k_{i}\neq 0$, $i=1,2,...,n-1$) in $E^{n}$ if and
only if $\sum\limits_{i=1}^{n-2}H_{i}^{2}=$constant and $H_{n-2}\neq 0$.
\end{theorem}

\begin{proof}
Let $\alpha $ be a inclined curve. According to the Definition 2.1,%
\begin{equation}
\left\langle V_{1},X\right\rangle =\cos (\varphi )=\text{constant,}
\end{equation}%
where $X$ the axis of $\alpha $. And, from Theorem (3.1),%
\begin{equation}
\left\langle V_{i+2},X\right\rangle =H_{i}\left\langle V_{1},X\right\rangle
\end{equation}%
for $1\leq i\leq n-2$. Moreover, from (3.10) and Frenet equations, we can
write $\left\langle V_{2},X\right\rangle =0$.Since the orthonormal system$%
\left\{ V_{1},V_{2},...,V_{n}\right\} $ is a basis of $\varkappa (E^{n})$
(tangent bundle), $X$ can be expressed in the form%
\begin{equation}
X=\sum\limits_{i=1}^{n}\left\langle V_{i},X\right\rangle V_{i}\text{.}
\end{equation}%
Hence, by using the equations (3.10), (3.11) and (3.12), we obtain%
\begin{equation*}
X=\cos (\varphi )V_{1}+\dsum\limits_{i=1}^{n-2}H_{i}\cos (\varphi )V_{i+2}%
\text{.}
\end{equation*}%
Since $X$ is a unit vector field (see Definition 2.1),%
\begin{equation*}
\cos ^{2}(\varphi )+\dsum\limits_{i=1}^{n-2}H_{i}^{2}\cos ^{2}(\varphi )=1
\end{equation*}%
and so%
\begin{equation*}
\sum\limits_{i=1}^{n-2}H_{i}^{2}=\tan ^{2}(\varphi )=\text{constant.}
\end{equation*}%
Now, we are going to show that $H_{n-2}\neq 0$. We assume that $H_{n-2}=0$.
Then, for $i=n-2$ in Theorem 3.1,%
\begin{equation*}
\left\langle V_{n},X\right\rangle =H_{n-2}\left\langle V_{1},X\right\rangle
=0\text{.}
\end{equation*}%
So, $\left\langle D_{T}V_{n},X\right\rangle =\left\langle
-k_{n-1}V_{n-1},X\right\rangle =0$. We deduce that $\left\langle
V_{n-1},X\right\rangle =0$. On the other hand, for $i=n-3$ in Theorem 3.1,%
\begin{equation*}
\left\langle V_{n-1},X\right\rangle =H_{n-3}\left\langle
V_{1},X\right\rangle \text{.}
\end{equation*}%
And, since $\left\langle V_{n-1},X\right\rangle =0$, $H_{n-3}=0$. Continuing
this process, we get that $H_{1}=0$. Let us recall that $H_{1}=k_{1}/k_{2}$,
thus we have a contradiction because all the curvatures are nowhere zero.
Consequently $H_{n-2}\neq 0$.

Conversely, we assume that $\sum\limits_{i=1}^{n-2}H_{i}^{2}=\tan
^{2}(\varphi )=$constant and $H_{n-2}\neq 0$. Then, consider the vector field%
\begin{equation*}
X=\cos (\varphi )V_{1}+\sum\limits_{i=3}^{n}H_{i-2}\cos (\varphi )V_{i}\text{%
.}
\end{equation*}%
We want to verify that $X$ is a constant along $\alpha $, i.e. $D_{V_{1}}X=0$%
. So,%
\begin{eqnarray*}
D_{V_{1}}X &=&D_{V_{1}}\left( \cos (\varphi )V_{1}\right)
+\sum\limits_{i=3}^{n}D_{V_{1}}\left( H_{i-2}\cos (\varphi )V_{i}\right) \\
&=&\cos (\varphi )D_{V_{1}}V_{1}+\sum\limits_{i=3}^{n}\left( H_{i-2}^{%
{\acute{}}%
}\cos (\varphi )V_{i}+H_{i-2}\cos (\varphi )D_{V_{1}}V_{i}\right) \\
&=&\cos (\varphi )\left( k_{1}V_{2}+\sum\limits_{i=3}^{n-1}\left( H_{i-2}^{%
{\acute{}}%
}V_{i}-k_{i-1}H_{i-2}V_{i-1}+k_{i}H_{i-2}V_{i+1}\right) +H_{n-2}^{%
{\acute{}}%
}V_{n}-k_{n-1}H_{n-2}V_{n-1}\right) \text{.}
\end{eqnarray*}%
On the other hand, by using (3.2), we can write%
\begin{equation}
H_{i-2}^{%
{\acute{}}%
}=k_{i}H_{i-1}-k_{i-1}H_{i-3}
\end{equation}%
for $4\leq i\leq n-1$ together with (3.3). Moreover, from Lemma 3.1, we know
that%
\begin{equation}
H_{n-2}^{%
{\acute{}}%
}=-k_{n-1}H_{n-3}
\end{equation}%
Therefore, by using (3.3), (3.13) and (3.14) , an algebraic calculus shows
that $D_{V_{1}}X=0$. Since%
\begin{eqnarray*}
\left\Vert X\right\Vert &=&\cos ^{2}(\varphi
)+\sum\limits_{i=3}^{n}H_{i-2}^{2}\cos ^{2}(\varphi ) \\
&=&\cos ^{2}(\varphi )\left( 1+\sum\limits_{i=1}^{n-2}H_{i}^{2}\right) \\
&=&\cos ^{2}(\varphi )\left( 1+\tan ^{2}(\varphi )\right) \\
&=&1\text{ ,}
\end{eqnarray*}%
$X$ is a unit vector field. Furthermore, $\left\langle V_{1},X\right\rangle
=\cos (\varphi )=$constant. Hence, we deduce that $\alpha $ is an inclined
curve.
\end{proof}

\noindent \textbf{Remark 3.1. }The following corollary is the
reconfiguration of the Theorem 3.4 in \cite{cam}.

\begin{corollary}
Let $\left\{ V_{1},V_{2},...,V_{n}\right\} $ be the Frenet frame of the
curve $\alpha $ and let $\left\{ H_{1},H_{2},...,H_{n-2}\right\} $ be the
harmonic curvature functions of $\alpha $. Then, $\alpha $ is an inclined
curve (with the curvatures $k_{i}\neq 0$, $i=1,2,...,n-1$) in $E^{n}$ if and
only if $H_{n-2}^{%
{\acute{}}%
}=-k_{n-1}H_{n-3}$ and $H_{n-2}\neq 0$.
\end{corollary}

\begin{proof}
It is obvious by using Lemma (3.1) and Theorem (3.2).
\end{proof}

\section{$V_{n}$-slant helices and their harmonic curvature functions}

In this section, we reconfigure some known characterizations by using
harmonic curvatures for $V_{n}$-slant helices.

\begin{definition}
Let $\alpha :I\subset 
\mathbb{R}
\rightarrow E^{n}$ be a unit speed curve with nonzero curvatures $k_{i}$ ($%
i=1,2,...,n-1$) in $E^{n}$. Harmonic curvature functions of $\alpha $ are
defined by $H_{i}^{\ast }:I\subset 
\mathbb{R}
\rightarrow 
\mathbb{R}
$,%
\begin{equation*}
H_{0}^{\ast }=0,H_{1}^{\ast }=\frac{k_{n-1}}{k_{n-2}},H_{i}^{\ast }=\left\{
k_{n-i}H_{i-2}^{\ast }-H_{i-1}^{\ast \prime }\right\} \frac{1}{k_{n-(i+1)}}
\end{equation*}%
for $2\leq i\leq n-2$ \cite{gok}.
\end{definition}

\begin{lemma}
Let $\alpha $ be a unit curve in $E^{n}$ and let $H_{n-2}^{\ast }\neq 0$ be
for $i=n-2$ . Then, $H_{1}^{\ast 2}+H_{2}^{\ast 2}+...+H_{n-2}^{\ast 2}$ is
a nonzero constant if and only if $H_{n-2}^{\ast \prime }=k_{1}H_{n-3}^{\ast
}$.
\end{lemma}

\begin{proof}
First, we assume that $H_{1}^{\ast 2}+H_{2}^{\ast 2}+...+H_{n-2}^{\ast 2}$
is a nonzero constant . Consider the functions%
\begin{equation*}
H_{i}^{\ast }=\left\{ k_{n-i}H_{i-2}^{\ast }-H_{i-1}^{\ast \prime }\right\} 
\frac{1}{k_{n-(i+1)}}
\end{equation*}%
for $3\leq i\leq n-2$. So, from the equality, we can write%
\begin{equation}
k_{n-(i+1)}H_{i}^{\ast }=k_{n-i}H_{i-2}^{\ast }-H_{i-1}^{\ast \prime }\text{%
, }3\leq i\leq n-2\text{.}
\end{equation}%
Hence, in (4.1), if we take $i+1$ instead of $i$, we get%
\begin{equation}
H_{i}^{\ast \prime }=k_{n-(i+1)}H_{i-1}^{\ast }-k_{n-(i+2)}H_{i+1}^{\ast
},2\leq i\leq n-3
\end{equation}%
together with%
\begin{equation}
H_{1}^{\ast \prime }=-k_{n-3}H_{2}^{\ast }\text{.}
\end{equation}%
On the other hand, since $H_{1}^{\ast 2}+H_{2}^{\ast 2}+...+H_{n-2}^{\ast 2}$
is constant, we have%
\begin{equation*}
H_{1}^{\ast }H_{1}^{\ast \prime }+H_{2}^{\ast }H_{2}^{\ast \prime
}+...+H_{n-2}^{\ast }H_{n-2}^{\ast \prime }=0
\end{equation*}%
and so, 
\begin{equation}
H_{n-2}^{\ast }H_{n-2}^{\ast \prime }=-H_{1}^{\ast }H_{1}^{\ast \prime
}-H_{2}^{\ast }H_{2}^{\ast \prime }-...-H_{n-3}^{\ast }H_{n-3}^{\ast \prime }%
\text{.}
\end{equation}%
By using (4.2) and (4.3), we obtain%
\begin{equation}
H_{1}^{\ast }H_{1}^{\ast \prime }=-k_{n-3}H_{1}^{\ast }H_{2}^{\ast }
\end{equation}%
and%
\begin{equation}
H_{i}^{\ast }H_{i}^{\ast \prime }=k_{n-(i+1)}H_{i-1}^{\ast }H_{i}^{\ast
}-k_{n-(i+2)}H_{i}^{\ast }H_{i+1}^{\ast }\text{, }2\leq i\leq n-3\text{.}
\end{equation}%
Therefore, by using (4.4), (4.5) and (4.6), a algebraic calculus shows that%
\begin{equation*}
H_{n-2}^{\ast }H_{n-2}^{\ast \prime }=k_{1}H_{n-3}^{\ast }H_{n-2}^{\ast }%
\text{.}
\end{equation*}%
Since $H_{n-2}^{\ast }\neq 0$, we get the relation $H_{n-2}^{\ast \prime
}=k_{1}H_{n-3}^{\ast }$.

Conversely, we assume that%
\begin{equation}
H_{n-2}^{\ast \prime }=k_{1}H_{n-3}^{\ast }.
\end{equation}%
By using (4.7) and $H_{n-2}^{\ast }\neq 0$, we can write%
\begin{equation}
H_{n-2}^{\ast }H_{n-2}^{\ast \prime }=k_{1}H_{n-2}^{\ast }H_{n-3}^{\ast }
\end{equation}%
From (4.6), we have 
\begin{eqnarray*}
\text{for }i &=&n-3\text{, \ \ \ \ \ }H_{n-3}^{\ast }H_{n-3}^{\ast \prime
}=k_{2}H_{n-4}^{\ast }H_{n-3}^{\ast }-k_{1}H_{n-3}^{\ast }H_{n-2}^{\ast } \\
\text{for }i &=&n-4\text{, \ \ \ \ \ }H_{n-4}^{\ast }H_{n-4}^{\ast \prime
}=k_{3}H_{n-5}^{\ast }H_{n-4}^{\ast }-k_{2}H_{n-4}^{\ast }H_{n-3}^{\ast } \\
\text{for }i &=&n-5\text{, \ \ \ \ \ }H_{n-5}^{\ast }H_{n-5}^{\ast \prime
}=k_{4}H_{n-6}^{\ast }H_{n-5}^{\ast }-k_{3}H_{n-5}^{\ast }H_{n-4}^{\ast } \\
&&\cdot  \\
&&\cdot  \\
&&\cdot  \\
\text{for }i &=&2\text{, \ \ \ \ \ \ \ \ \ \ }H_{2}^{\ast }H_{2}^{\ast
\prime }=k_{n-3}H_{1}^{\ast }H_{2}^{\ast }-k_{n-4}H_{2}^{\ast }H_{3}^{\ast }
\end{eqnarray*}%
and from (4.5), we have%
\begin{equation*}
H_{1}^{\ast }H_{1}^{\ast \prime }=-k_{n-3}H_{1}^{\ast }H_{2}^{\ast }\text{.}
\end{equation*}%
So, an algebraic calculus show that%
\begin{equation}
H_{1}^{\ast }H_{1}^{\ast \prime }+H_{2}^{\ast }H_{2}^{\ast \prime }+...+%
\text{\ }H_{n-5}^{\ast }H_{n-5}^{\ast \prime }+H_{n-4}^{\ast }H_{n-4}^{\ast
\prime }+H_{n-3}^{\ast }H_{n-3}^{\ast \prime }+H_{n-2}^{\ast }H_{n-2}^{\ast
\prime }=0\text{.}
\end{equation}%
And, by integrating (4.9), we can easily say that%
\begin{equation*}
H_{1}^{\ast 2}+H_{2}^{\ast 2}+...+H_{n-2}^{\ast 2}
\end{equation*}%
is a nonzero constant. This completes the proof.
\end{proof}

\noindent \textbf{Proposition 4.1. }Let $\alpha :I\subset 
\mathbb{R}
\rightarrow E^{n}$ be an arc-lengthed parameter curve in $E^{n}$ and $X$ a
unit constant vector field of $%
\mathbb{R}
^{n}$. $\left\{ V_{1},V_{2},...,V_{n}\right\} $ denote the Frenet frame of
the curve $\alpha $ and $\left\{ H_{1}^{\ast },H_{2}^{\ast
},...,H_{n-2}^{\ast }\right\} $ denote the harmonic curvature functions of
the curve $\alpha $. If $\alpha :I\subset 
\mathbb{R}
\rightarrow E^{n}$ is an $V_{n}$-slant helix with $X$ as its axis , then we
have for all $i=0,1,...,n-2$%
\begin{equation*}
\left\langle V_{n-(i+1)},X\right\rangle =H_{i}^{\ast }\left\langle
V_{n},X\right\rangle \text{.}
\end{equation*}%
\cite{gok}.

\noindent \textbf{Remark 4.1. }The following Theorem is the new version of
the Theorem 4 in \cite{gok} with addition sufficiency case.

\begin{theorem}
Let $\left\{ V_{1},V_{2},...,V_{n}\right\} $ be the Frenet frame of the
curve $\alpha $ and let $\left\{ H_{1}^{\ast },H_{2}^{\ast
},...,H_{n-2}^{\ast }\right\} $ be the harmonic curvature functions of $%
\alpha $. Then, $\alpha $ is an $V_{n}$ -slant helix (with the curvatures $%
k_{i}\neq 0$, $i=1,2,...,n-1$) in $E^{n}$ if and only if $%
\sum\limits_{i=1}^{n-2}H_{i}^{\ast 2}=$constant and $H_{n-2}^{\ast }\neq 0$.
\end{theorem}

\begin{proof}
Let $\alpha $ be a $V_{n}$ -slant helix . According to the Definition 2.3,%
\begin{equation}
\left\langle V_{n},X\right\rangle =\cos (\varphi )=\text{constant,}
\end{equation}%
where $X$ the axis of $\alpha $. And, from Proposition 4.1., 
\begin{equation}
\left\langle V_{n-(i+1)},X\right\rangle =H_{i}^{\ast }\left\langle
V_{n},X\right\rangle \text{.}
\end{equation}%
for $1\leq i\leq n-2$. Moreover, from (4.10) and Frenet equations, we can
write $\left\langle V_{n-1},X\right\rangle =0$.Since the orthonormal system$%
\left\{ V_{1},V_{2},...,V_{n}\right\} $ is a basis of $\varkappa (E^{n})$
(tangent bundle), $X$ can be expressed in the form%
\begin{equation}
X=\sum\limits_{i=1}^{n}\left\langle V_{i},X\right\rangle V_{i}\text{.}
\end{equation}%
Hence, by using the equations (4.10), (4.11) and (4.12), we obtain%
\begin{equation*}
X=\cos (\varphi )V_{n}+\dsum\limits_{i=1}^{n-2}H_{i}^{\ast }\cos (\varphi
)V_{n-(i+1)}\text{.}
\end{equation*}%
Since $X$ is a unit vector field (see Definition 2.3), 
\begin{equation*}
\cos ^{2}(\varphi )+\dsum\limits_{i=1}^{n-2}H_{i}^{\ast 2}\cos ^{2}(\varphi
)=1
\end{equation*}%
and so%
\begin{equation*}
\sum\limits_{i=1}^{n-2}H_{i}^{\ast 2}=\tan ^{2}(\varphi )=\text{constant.}
\end{equation*}%
Now, we are going to show that $H_{n-2}^{\ast }\neq 0$. We assume that $%
H_{n-2}^{\ast }=0$. Then, for $i=n-2$ in (4.11), 
\begin{equation*}
\left\langle V_{1},X\right\rangle =H_{n-2}^{\ast }\left\langle
V_{n},X\right\rangle =0\text{.}
\end{equation*}%
So, $\left\langle D_{T}T,X\right\rangle =\left\langle
k_{1}V_{2},X\right\rangle =0$. We deduce that $\left\langle
V_{2},X\right\rangle =0$. On the other hand, for $i=n-3$ in (4.11),%
\begin{equation*}
\left\langle V_{2},X\right\rangle =H_{n-3}^{\ast }\left\langle
V_{n},X\right\rangle \text{.}
\end{equation*}%
And, since $\left\langle V_{2},X\right\rangle =0$, $H_{n-3}^{\ast }=0$.
Continuing this process, we get that $H_{1}^{\ast }=0$. Let us recall that $%
H_{1}^{\ast }=k_{n-1}/k_{n-2}$, thus we have a contradiction because all the
curvatures are nowhere zero. Consequently $H_{n-2}^{\ast }\neq 0$.

Conversely, we assume that $\sum\limits_{i=1}^{n-2}H_{i}^{\ast 2}=\tan
^{2}(\varphi )=$constant and $H_{n-2}^{\ast }\neq 0$. Then, consider the
vector field%
\begin{equation*}
X=\cos (\varphi )V_{n}+\sum\limits_{i=3}^{n}H_{i-2}^{\ast }\cos (\varphi
)V_{n-(i-1)}\text{.}
\end{equation*}%
We want to verify that $X$ is a constant along $\alpha $, i.e. $D_{V_{1}}X=0$%
. So,%
\begin{eqnarray*}
D_{V_{1}}X &=&D_{V_{1}}\left( \cos (\varphi )V_{n}\right)
+\sum\limits_{i=3}^{n}D_{V_{1}}\left( H_{i-2}^{\ast }\cos (\varphi
)V_{n-(i-1)}\right) \\
&=&\cos (\varphi )D_{V_{1}}V_{n}+\sum\limits_{i=3}^{n}\left( H_{i-2}^{\ast
\prime }\cos (\varphi )V_{n-(i-1)}+H_{i-2}^{\ast }\cos (\varphi
)D_{V_{1}}V_{n-(i-1)}\right) \\
&=&\cos (\varphi )(-k_{n-1}V_{n-1}+\sum\limits_{i=3}^{n-1}\left(
H_{i-2}^{\ast \prime }V_{n-(i-1)}-k_{n-i}H_{i-2}^{\ast
}V_{n-i}+k_{n-(i-1)}H_{i-2}^{\ast }V_{n-(i-2)}\right) + \\
&&H_{n-2}^{\ast \prime }V_{1}+k_{1}H_{n-2}^{\ast }V_{2})\text{.}
\end{eqnarray*}%
On the other hand, by using (4.2), we can write 
\begin{equation}
H_{i-2}^{\ast \prime }=k_{n-(i-1)}H_{i-3}^{\ast }-k_{n-i}H_{i-1}^{\ast }
\end{equation}%
for $4\leq i\leq n-1$ together with (4.3). Moreover, from Lemma 4.1, we know
that%
\begin{equation}
H_{n-2}^{\ast \prime }=k_{1}H_{n-3}^{\ast }.
\end{equation}%
Therefore, by using (4.3), (4.13) and (4.14) , an algebraic calculus shows
that $D_{V_{1}}X=0$. Since%
\begin{eqnarray*}
\left\Vert X\right\Vert &=&\cos ^{2}(\varphi
)+\sum\limits_{i=3}^{n}H_{i-2}^{\ast 2}\cos ^{2}(\varphi ) \\
&=&\cos ^{2}(\varphi )\left( 1+\sum\limits_{i=1}^{n-2}H_{i}^{\ast 2}\right)
\\
&=&\cos ^{2}(\varphi )\left( 1+\tan ^{2}(\varphi )\right) \\
&=&1\text{ ,}
\end{eqnarray*}%
$X$ is a unit vector field. Furthermore, $\left\langle V_{n},X\right\rangle
=\cos (\varphi )=$constant. Hence, we deduce that $\alpha $ is a $V_{n}$%
.-slant helix.
\end{proof}

\noindent \textbf{Remark 4.2. }The following corollary is the
reconfiguration of the Theorem 2 in \cite{gok}.

\begin{corollary}
Let $\left\{ V_{1},V_{2},...,V_{n}\right\} $ be the Frenet frame of the
curve $\alpha $ and let $\left\{ H_{1}^{\ast },H_{2}^{\ast
},...,H_{n-2}^{\ast }\right\} $ be the harmonic curvature functions of $%
\alpha $. Then, $\alpha $ is a $V_{n}$-slant helix (with the curvatures $%
k_{i}\neq 0$, $i=1,2,...,n-1$) in $E^{n}$ if and only if $H_{n-2}^{\ast
\prime }=k_{1}H_{n-3}^{\ast }$ and $H_{n-2}^{\ast }\neq 0$.
\end{corollary}

\begin{proof}
It is obvious by using Lemma 4.1. and Theorem 4.1.
\end{proof}

\section{Slant helices and their $G_{i}$ differentiable functions}

In this section, we reconfigure some known characterizations of slant
helices by using $G_{i}$ differentiable functions which is similar to
harmonic curvature functions.

\begin{definition}
Let $\alpha :I\rightarrow E^{n}$ be a unit speed curve (with the curvatures $%
k_{i}\neq 0$, $i=1,2,...,n-1$) in $E^{n}$. Define the functions 
\begin{equation}
G_{1}=\int k_{1}(s)ds\text{ , }G_{2}=1\text{ , }G_{3}=\frac{k_{1}}{k_{2}}%
G_{1}\text{ , }G_{i}=\frac{1}{k_{i-1}}\left[ k_{i-2}G_{i-2}+G_{i-1}^{\prime }%
\right]
\end{equation}%
where $4\leq i\leq n$ \cite{ahmad}.
\end{definition}

\begin{lemma}
Let $\alpha $ be a unit curve in $E^{n}$ and let $G_{n}\neq 0$ be for $i=n$
. Then, $G_{1}^{2}+G_{2}^{2}+...+G_{n}^{2}$ is a nonzero constant if and
only if $G_{n}^{%
{\acute{}}%
}=-k_{n-1}G_{n-1}$.
\end{lemma}

\begin{proof}
First, we assume that $G_{1}^{2}+G_{2}^{2}+...+G_{n}^{2}$ is a nonzero
constant. Consider the functions%
\begin{equation*}
G_{i}=\frac{1}{k_{i-1}}\left[ k_{i-2}G_{i-2}+G_{i-1}^{\prime }\right]
\end{equation*}%
for $5\leq i\leq n$. So, from the equality, we can write%
\begin{equation}
k_{i-1}G_{i}=G_{i-1}^{%
{\acute{}}%
}+k_{i-2}G_{i-2}\text{, }5\leq i\leq n\text{.}
\end{equation}%
Hence, in (5.2), if we take $i+1$ instead of $i$, we get%
\begin{equation}
G_{i}^{%
{\acute{}}%
}=k_{i}G_{i+1}-k_{i-1}G_{i-1}\text{, }4\leq i\leq n-1\text{.}
\end{equation}%
together with 
\begin{equation}
G_{1}=\int k_{1}(s)ds\text{ , }G_{2}=1\text{ , }G_{3}=\frac{k_{1}}{k_{2}}%
G_{1}
\end{equation}%
On the other hand, since $G_{1}^{2}+G_{2}^{2}+...+G_{n}^{2}$ is constant, we
have%
\begin{equation*}
G_{1}G_{1}^{%
{\acute{}}%
}+G_{2}G_{2}%
{\acute{}}%
+...+G_{n}G_{n}^{%
{\acute{}}%
}=0
\end{equation*}%
and so, 
\begin{equation}
G_{n}G_{n}^{%
{\acute{}}%
}=-G_{1}G_{1}^{%
{\acute{}}%
}-G_{2}G_{2}%
{\acute{}}%
-...-G_{n-1}G_{n-1}^{%
{\acute{}}%
}\text{.}
\end{equation}%
By using (5.3) and (5.4), we obtain%
\begin{equation}
G_{2}G_{2}^{\prime }=0\text{ \ and \ }k_{3}G_{3}G_{4}=G_{1}G_{1}^{\prime
}+G_{3}G_{3}^{\prime }
\end{equation}%
and%
\begin{equation}
G_{i}G_{i}^{%
{\acute{}}%
}=k_{i}G_{i}G_{i+1}-k_{i-1}G_{i-1}G_{i}\text{ , }4\leq i\leq n-1\text{.}
\end{equation}%
Therefore, by using (5.5), (5.6) and (5.7), a algebraic calculus shows that%
\begin{equation*}
G_{n}G_{n}^{%
{\acute{}}%
}=-k_{n-1}G_{n-1}G_{n}\text{ .}
\end{equation*}%
Since $G_{n}\neq 0$, we get the relation $G_{n}^{%
{\acute{}}%
}=-k_{n-1}G_{n-1}$.

Conversely, we assume that%
\begin{equation}
G_{n}^{%
{\acute{}}%
}=-k_{n-1}G_{n-1}\text{ .}
\end{equation}%
By using (5.8) and $G_{n}\neq 0$, we can write%
\begin{equation}
G_{n}G_{n}^{%
{\acute{}}%
}=-k_{n-1}G_{n-1}G_{n}\text{ .}
\end{equation}%
From (5.7), we have%
\begin{eqnarray*}
\text{for }i &=&n-1\text{, \ \ \ \ \ }G_{n-1}G_{n-1}^{%
{\acute{}}%
}=k_{n-1}G_{n-1}G_{n}-k_{n-2}G_{n-2}G_{n-1} \\
\text{for }i &=&n-2\text{, \ \ \ \ \ }G_{n-2}G_{n-2}^{%
{\acute{}}%
}=k_{n-2}G_{n-2}G_{n-1}-k_{n-3}G_{n-3}G_{n-2} \\
\text{for }i &=&n-3\text{, \ \ \ \ \ }G_{n-3}G_{n-3}^{%
{\acute{}}%
}=k_{n-3}G_{n-3}G_{n-2}-k_{n-4}G_{n-4}G_{n-3} \\
&&\cdot  \\
&&\cdot  \\
&&\cdot  \\
\text{for }i &=&4\text{, \ \ \ \ \ \ \ \ \ \ }G_{4}G_{4}^{%
{\acute{}}%
}=k_{4}G_{4}G_{5}-k_{3}G_{3}G_{4}
\end{eqnarray*}%
and so, from (5.9) and the last system, we have 
\begin{equation}
G_{4}G_{4}^{\prime }+G_{5}G_{5}^{\prime }+...+G_{n}G_{n}^{\prime
}=-k_{3}G_{3}G_{4}
\end{equation}%
by doing an algebraic calculus. On the other hand, from (5.6), we know that 
\begin{equation}
G_{2}G_{2}^{\prime }=0\text{ \ and \ }k_{3}G_{3}G_{4}=G_{1}G_{1}^{\prime
}+G_{3}G_{3}^{\prime }\text{ .}
\end{equation}%
Finally, from (5.10) and (5.11), we obtain%
\begin{equation}
G_{1}G_{1}^{%
{\acute{}}%
}+G_{2}G_{2}%
{\acute{}}%
+...+G_{n}G_{n}^{%
{\acute{}}%
}=0\text{ .}
\end{equation}%
And, by integrating (5.12), we can easily say that%
\begin{equation*}
G_{1}^{2}+G_{2}^{2}+...+G_{n}^{2}
\end{equation*}%
is a nonzero constant. This completes the proof.
\end{proof}

\begin{corollary}
Let $\alpha :I\subset 
\mathbb{R}
\rightarrow E^{n}$ be an arc-lengthed parameter curve with nonzero
curvatures $k_{i}$ ($1\leq i\leq n-1$) in $E^{n}$ and $X$ a unit constant
vector field of $%
\mathbb{R}
^{n}$. $\left\{ V_{1},V_{2},...,V_{n}\right\} $ denote the Frenet frame of
the curve $\alpha $. If $\alpha :I\subset 
\mathbb{R}
\rightarrow E^{n}$ is an $V_{2}$-slant helix with $X$ as its axis, then we
have for all $i=1,...,n$%
\begin{equation*}
\left\langle V_{i},X\right\rangle =G_{i}\left\langle V_{2},X\right\rangle .
\end{equation*}
\end{corollary}

\begin{proof}
It is obvious by using the proof of Theorem 1.2 in \cite{ahmad}.
\end{proof}

\noindent \textbf{Remark 5.1. }The following Theorem is the new version of
the Theorem 1.2 in \cite{ahmad}.

\begin{theorem}
Let $\left\{ V_{1},V_{2},...,V_{n}\right\} $ be the Frenet frame of the
curve $\alpha $ . Then, $\alpha $ is a $V_{2}$-slant helix (with the
curvatures $k_{i}\neq 0$, $i=1,2,...,n-1$) in $E^{n}$ if and only if $%
\sum\limits_{i=1}^{n}G_{i}^{2}=$constant and $G_{n}\neq 0$. Here, 
\begin{equation*}
G_{1}=\int k_{1}(s)ds\text{ , }G_{2}=1\text{ , }G_{3}=\frac{k_{1}}{k_{2}}%
G_{1}\text{ , }G_{i}=\frac{1}{k_{i-1}}\left[ k_{i-2}G_{i-2}+G_{i-1}^{\prime }%
\right]
\end{equation*}%
where $4\leq i\leq n$.
\end{theorem}

\begin{proof}
Let $\alpha $ be a $V_{2}$-slant helix. According to the Definition 2.2,%
\begin{equation}
\left\langle V_{2},X\right\rangle =\cos (\varphi )=\text{constant,}
\end{equation}%
where $X$ the axis of $\alpha $. And, from Corollary 5.1., 
\begin{equation}
\left\langle V_{i},X\right\rangle =G_{i}\left\langle V_{2},X\right\rangle
\end{equation}%
for $1\leq i\leq n$. Since the orthonormal system$\left\{
V_{1},V_{2},...,V_{n}\right\} $ is a basis of $\varkappa (E^{n})$ (tangent
bundle), $X$ can be expressed in the form%
\begin{equation}
X=\sum\limits_{i=1}^{n}\left\langle V_{i},X\right\rangle V_{i}\text{.}
\end{equation}%
Hence, by using the equations (5.13), (5.14) and (5.15), we obtain%
\begin{equation*}
X=\dsum\limits_{i=1}^{n}G_{i}\cos (\varphi )V_{i}\text{.}
\end{equation*}%
Since $X$ is a unit vector field (see Definition 2.2),%
\begin{equation*}
\cos ^{2}(\varphi )\left( \dsum\limits_{i=1}^{n}G_{i}^{2}\right) =1
\end{equation*}%
and so, 
\begin{equation*}
\dsum\limits_{i=1}^{n}G_{i}^{2}=\frac{1}{\cos ^{2}(\varphi )}=\text{constant
.}
\end{equation*}%
Now, we are going to show that $G_{n}\neq 0$. We assume that $G_{n}=0$.
Then, for $i=n$ in (5.14),%
\begin{equation*}
\left\langle V_{n},X\right\rangle =G_{n}\left\langle V_{2},X\right\rangle =0%
\text{.}
\end{equation*}%
So, $\left\langle D_{T}V_{n},X\right\rangle =\left\langle
-k_{n-1}V_{n-1},X\right\rangle =0$. We deduce that $\left\langle
V_{n-1},X\right\rangle =0$. On the other hand, for $i=n-1$ in (5.14),%
\begin{equation*}
\left\langle V_{n-1},X\right\rangle =G_{n-1}\left\langle
V_{2},X\right\rangle \text{.}
\end{equation*}%
And, since $\left\langle V_{n-1},X\right\rangle =0$, $G_{n-1}=0$. Continuing
this process, we get that $G_{3}=0$. Let us recall that $G_{3}=\dfrac{k_{1}}{%
k_{2}}\int k_{1}\left( s\right) ds$, thus we have a contradiction because
all the curvatures are nowhere zero. Consequently $G_{n}\neq 0$.

Conversely, we assume that $\sum\limits_{i=1}^{n}G_{i}^{2}=\dfrac{1}{\cos
^{2}(\varphi )}=$constant and $G_{n}\neq 0$. Then, consider the vector field%
\begin{equation*}
X=\dsum\limits_{i=1}^{n}G_{i}\cos (\varphi )V_{i}\text{.}
\end{equation*}%
Then, by taking account 
\begin{equation*}
G_{1}=\int k_{1}(s)ds\text{ , }G_{2}=1\text{ , }G_{3}=\frac{k_{1}}{k_{2}}%
G_{1}\text{ , }G_{i}=\frac{1}{k_{i-1}}\left[ k_{i-2}G_{i-2}+G_{i-1}^{\prime }%
\right] \text{ , }4\leq i\leq n
\end{equation*}%
and Frenet equations, an algebraic calculus shows that $D_{V_{1}}X=0$. That
is, $X$ is a constant along $\alpha $. Also, since 
\begin{eqnarray*}
\left\Vert X\right\Vert &=&\sum\limits_{i=1}^{n}G_{i}^{2}\cos ^{2}(\varphi )
\\
&=&\cos ^{2}(\varphi )\left( \sum\limits_{i=1}^{n}G_{i}^{2}\right) \\
&=&\cos ^{2}(\varphi )\frac{1}{\cos ^{2}(\varphi )} \\
&=&1\text{ ,}
\end{eqnarray*}%
$X$ is a unit vector field. Furthermore, $\left\langle V_{2},X\right\rangle
=\cos (\varphi )=$constant. Hence, we deduce that $\alpha $ is a $V_{2}$%
.-slant helix.
\end{proof}

\noindent \textbf{Remark 5.2. }The following corollary is the
reconfiguration of the Theorem 3.1 in \cite{ahmad}.

\begin{corollary}
Let $\left\{ V_{1},V_{2},...,V_{n}\right\} $ be the Frenet frame of the
curve $\alpha $ . Then, $\alpha $ is a $V_{2}$-slant helix in $E^{n}$ if and
only if $G_{n}^{%
{\acute{}}%
}=-k_{n-1}G_{n-1}$ and $G_{n}\neq 0$, where the functions $\left\{
G_{1},G_{2},...,G_{n}\right\} $ defined in (5.1).
\end{corollary}

\begin{proof}
It is obvious by using Lemma 5.1. and Theorem 5.1.
\end{proof}

\textbf{Acknowledgment.}

The last author would like to thank Tubitak-Bideb for their financial
supports during his PhD studies and the authors thank Proffessor H. H. Hac\i
saliho\u{g}lu for several useful remarks on this paper.

\end{document}